\documentclass[11pt,dvips,twoside,letterpaper]{article} 
\usepackage{amssymb,amsthm,amsmath,amscd,amsfonts,txfonts}
\usepackage{xcolor}
\usepackage{hyperref}

\newtheorem{remark}{Remark}[section]

\theoremstyle{plain}
\newtheorem{lemma}{Lemma}[section]
\newtheorem{proposition}{Proposition}[section]
\newtheorem{theorem}{Theorem}[section]
\newtheorem{definition}{Definition}[section]
\newtheorem{corollary}{Corollary}[section]

\newtheorem{example}{Example}[section]

\title{\bfseries\scshape{Hom-alternative  modules and Hom-Poisson  comodules}}
\author{\bfseries\scshape  Ibrahima BAKAYOKO \thanks{E-mail address: \tt{ibrahimabakayoko@yahoo.fr}}\\
 D\'epartement de Math\'ematiques,\\ 
UJNK/Centre Universitaire de N'Z\'er\'ekor\'e,\\
BP : 50, N'Z\'er\'ekor\'e, Guin\'ee.\\
\\\bfseries\scshape Bakary MANGA\thanks{E-mail address: \tt{bakary.manga@ucad.edu.sn, bakary.manga@imsp-uac.org}}\\
D\'epartement de Math\'ematiques et Informatique, \\ 
 Universit\'e Cheikh Anta Diop de Dakar, S\'en\'egal
\\ and\\
Institut de Math\'ematiques et de Sciences Physiques, Porto-Novo, B\'enin.}

\date{} 

\begin{document} 
\maketitle

\begin{abstract} 
In this paper we introduce modules over both left and right Hom-alternative algebras. 
We give some constructions of left and right Hom-alternative modules and give various 
properties of both, as well as examples. Then, we prove that morphisms of left alternative 
algebras extend to morphisms of left Hom-alternative algebras. 
Next, we introduce comodules over Hom-Poisson coalgebras and show that we may obtain a structure map 
of a comodule over a Hom-Poisson coalgebra from a given one.
\end{abstract} 



\section{Introduction}

The study of Hom-algebraic structures, that is algebras where the identities defining the structures are twisted by a homomorphism, 
is an active and prolific field of research. The main motivations of studying such structures come from deformations of Witt and Virasoro 
algebras (\cite{hartwig-larson-silvestrov:deformations}, \cite{larson-silvestrov:quasi-hom-lie-central-extensions}), 
number theory, Yang-Baxter Equations, quantum groups, etc (see \cite{yau-hom-bialgebras-comodule-hom-bialgebras}
and references therein). There are various settings of Hom-structures such as algebras, coalgebras, Hopf algebras, Leibniz algebras, 
$n$-ary algebras, Poisson algebras, alternative algebras, etc (see \cite{elhamdadi-makhlouf:deformations-hom-alternative-hom-malcev}
and references therein). 

Left alternative algebras, right alternative algebras and Hom-alternative algebras were introduced by Makhlouf 
in  \cite{makhlouf:hom-alternative-hom-jordan}. Since then, twisted versions of alternative algebras arise: the so-called Hom-alternative
algebras. 
% 
% A construction of Hom-alternative
% algebras are given ; ordinary alternative algebras lead to Hom-alternative algebras via an algebra endomorphism. Examples of non-associative
% Hom-alternative algebras are derived from $4$-dimensional alternative algebras and from octinions algebra.
Hom-Poisson algebras, when to them, were introduced in \cite{AS2} where they appear naturally in the study 
of $1$-parameter formal deformations of commutative Hom-associative algebras.
Noncommutative Hom-Poisson algebras are a twisted generalization of noncommutative Poisson algebras. 
In mathematics, Poisson algebras play a fundamental role in Poisson geometry
 (\cite{vaisman-lecture-geometry-poisson}), quantum groups (\cite{drinfeld:quantum-group}), 
 and deformation of  commutative associative algebras (\cite{gerstenhaber-deformation-ring-algebra}). 
 In physics, Poisson algebras are a  major part of deformation quantization 
 (\cite{kontsevich:deformation-quatization-poisson}), 
 Hamiltonian mechanics (\cite{arnold-math-method-classical-mech}), and topological field theories (\cite{shaller-strobl:poisson-structure}). 
Poisson-like structures are also used in the study of vertex operator algebras (\cite{frenkel-ben_zvi:vertex}).
 Many examples of noncommutative  Hom-Poisson algebras are exposed in \cite{yau:noncom-hom-poisson-algebras}.
Modules over color Hom-Poisson algebras are studied in \cite{bakayoko:module-color-hom-poisson}.

As in the case of Hom-associative algebras, Hom-coassociative coalgebras, Hom-Lie algebras, Hom-bialgebras, and  
Hom-Lie bialgebras, Hom-Poisson coalgebras have a dual version called Hom-coPoisson algebras 
in \cite{bordeman-elchinger-makhlouf:twisting-poisson-copoisson-quantization}, where they were
 introduced and studied. It is shown that starting from classical Poisson coalgebras, 
 we may construct Hom-Poisson coalgebras using Hom-Poisson morphisms. Otherwise, a Hom-Poisson
coalgebra give rise to infinitely many Hom-Poisson coalgebras.

The purpose of this paper is to study and present various twisting
% {\color{red} Yau's twisting}
 of modules over left Hom-alternative algebras and comodules 
over Hom-Poisson coalgebras. In Section \ref{sec:preliminaries}, we recall basic definitions and properties on left and 
right Hom-alternative algebras. 
We give a connection between left alternative algebras and right alternative algebras.
Modules over left and right Hom-alternative algebras are introduced in Section \ref{sec:modules-alternative-algebras} where we 
also prove that twisting a left Hom-alternative module structure by an algebra endomorphism we get another one. 
We also show that morphisms of left alternative modules extend to morphisms of left
Hom-alternative modules. We introduce a version of Hom-associator for module and show that the Hom-alternative 
module is  alternating with right to its 
first two arguments. 
%Other properties of left and right Hom-alternative modules are also estblished.
Section \ref{sec:hom-poisson-comodules} deals with Hom-Poisson comodules. Some constructions of 
Hom-Poisson comodules are given and we carry out that starting 
from a Hom-Poisson comodule we get another one by twisting the module structure map by an algebra endomorphism.

Throughout this paper, $\bf K$ denotes a field of characteristic $0$ so that unless otherwise specified linearity and $\otimes$
are all meant over $\bf K$.

\section{Preliminaries \label{sec:preliminaries}}

\subsection{Hom-alternative algebras}
\begin{definition}[\cite{makhlouf:hom-alternative-hom-jordan}]
 A left (resp. right) Hom-alternative algebra  is a triple $(A, \mu, \alpha)$ consisting of a 
 $\bf K$-linear space $A$, a multiplication $\mu :A\otimes A\to A$ and a linear map 
 $\alpha : A\to A$ satisfying, for any $x, y\in A$, 
 the left (resp. right) Hom-alternative identity, that is 
 \begin{equation}
\mu\Big(\alpha(x), \mu(x, y)\Big)=\mu\Big(\mu(x, x), \alpha(y)\Big)  \quad \Big[\text{resp. } 
\mu\Big(\alpha(x), \mu(y, y)\Big)=\mu\Big(\mu(x, y),\alpha(y)\Big) \Big].
 \end{equation}
A Hom-alternative algebra is one which is both a left and a right Hom-alternative algebra.
\end{definition}
\begin{example}{\normalfont
A Hom-associative algebra is a triple  $(A,\mu,\alpha)$, where $\mu:A\times A\to A$ is
a bilinear map and $\alpha:A\to A$ is a linear map satisfying %: for any $x$, $y$ and $z$ in $A$,
$%\begin{equation}
 \mu\Big(\alpha(x),\mu(y,z)\Big)=\mu\Big(\mu(x,y),\alpha(z)\Big)
$, for any $x$, $y$ and $z$ in $A$. %\end{equation}
It is a little matter to verify that a Hom-associative algebra is a Hom-alternative algebra. }
\end{example}
Let $(A, \mu, \alpha)$ be a left Hom-alternative algebra. We consider the  maps 
$\tau:A\otimes A \to A\otimes A$ and $-\mu:A\otimes A\to A$ given respectively by $\tau(x,y)=(y,x)$ and $(-\mu)(x,y)=-(\mu(x,y))$, 
for all $x$ and $y$ in $A$. Now we define, the multiplication $\mu^{op}:A\otimes A\to A$ by $\mu^{op}=\mu\circ\tau$; that is 
$\mu^{op}(x, y)=\mu(y, x)$, for any $x,y \in A$. We also set $\mu_k=k\mu$ and $\alpha_k=k\alpha$, for all $k\in \bf K$.
We have the following straightforward result.
\begin{lemma}\label{-t}
 Let $(A,\mu,\alpha)$ be a left Hom-alternative algebra. Then, \\
%\begin{enumerate}
a)  $(A, \mu_k, \alpha_{k'})$ is a left Hom-alternative algebra, for all $k, k'\in \bf K$;\\
b)  $(A, \mu^{op}, \alpha)$ is a right Hom-alternative algebra.
%\end{enumerate}
% $(A, -\mu, \alpha)$ is also a left Hom-alternative algebra while 
\end{lemma}
\begin{definition}[\cite{makhlouf:hom-alternative-hom-jordan}]
 Let $(A, \mu, \alpha)$ and $(A', \mu', \alpha')$ be two Hom-alternative algebras. 
 A linear map $f : A\to A'$ is said to be a morphism of Hom-alternative algebras if 
 \begin{equation}
\mu'\circ(f\otimes f)=f\circ\mu \text{ and } f\circ\alpha=\alpha'\circ f.
 \end{equation}
\end{definition}
A usefull characterization of left (resp. right) Hom-alternative algebras is given by the
\begin{proposition}[\cite{makhlouf:hom-alternative-hom-jordan}] \label{pro:cns-hom-alternative-algebra}
 A triple $(A, \mu, \alpha)$ is a left (resp. right) Hom-alternative algebra if and only if the following identity holds:
 for any $x$, $y$ and $z$ in $A$, 
\begin{eqnarray}
 \mu(\alpha(x), \mu(y, z))-\mu(\mu(x, y), \alpha(z))+\mu(\alpha(y), \mu(x, z))-\mu(\mu(y, x), \alpha(z))=0,
\end{eqnarray}
respectively,
\begin{eqnarray}
 \mu(\alpha(x), \mu(y, z))-\mu(\mu(x, y),\alpha(z))+\mu(\alpha(x), \mu(z, y))-\mu(\mu(x, z), \alpha(y))=0.
\end{eqnarray}
\end{proposition}
\begin{proposition}[\cite{makhlouf:hom-alternative-hom-jordan}]
 Let $(A, \mu, \alpha)$ be a Hom-alternative algebra and $x, y, z\in A$. If $x$ and $y$ anticommute, 
 that is $\mu(x, y)=-\mu(y, x)$, then we have
\begin{eqnarray}
 \mu\Big(\alpha(x), \mu(y, z)\Big)=-\mu\Big(\alpha(y), \mu(x, z)\Big)\; \text{ and } \;  
 \mu\Big(\mu(z, x), \alpha(y)\Big)=-\mu\Big(\mu(z, y), \alpha(x)\Big).
\end{eqnarray}
\end{proposition}
Now given a left (resp. right) alternative algebra, one can construct a left (resp. right) Hom-alternative
algebra (see \cite{makhlouf:hom-alternative-hom-jordan}). Recall that an alternative algebra is a pair 
$(A,\mu)$, where $A$ is an algebra and  $\mu:A\times A\to A$ is an alternative 
(not necessarily associative) multiplication; {\it i. e.}  $\mu(x,\mu(x,y))=\mu(\mu(x,x),y)$, for any $x$ 
and $y$ in $A$.
\begin{lemma}[\cite{makhlouf:hom-alternative-hom-jordan}]\label{lhatl}
 Let $(A, \mu)$ be a left (resp. right) alternative algebra  and $\alpha : A\to A$ be an algebra endomorphism.
 Then $(A, \mu_\alpha, \alpha)$, where $\mu_\alpha=\alpha\circ\mu$, is a left (resp. right) Hom-alternative algebra.

 Moreover, suppose that $(A', \mu')$ is another left  (resp. right) alternative algebra and $\alpha' : A'\to A'$
is an algebra endomorphism. If $f : A\to A'$ is an algebras morphism that satisfies $\alpha'\circ f=f\circ\alpha$ then, 
$f : (A, \mu_\alpha, \alpha)\to (A', \mu'_{\alpha'}, \alpha')$
is a morphism of left (resp. right) Hom-alternative algebras.
\end{lemma}
\subsection{Hom-Poisson coalgebras}
In the sequel, $A$ is a $\bf K$-linear space while $\sigma$ denotes 
the linear map $\sigma: A\otimes A\otimes A \to A\otimes A\otimes A$
defined by $\sigma(x_1\otimes x_2\otimes x_3) = x_3\otimes x_1\otimes x_2$. 
\begin{definition}[\cite{bordeman-elchinger-makhlouf:twisting-poisson-copoisson-quantization}]\label{hpdd}
A cocommutative Hom-Poisson coalgebra consists of a quadruple $(A, \Delta, \gamma, \alpha)$ in which
\begin{enumerate}
 \item[i)] $(A, \Delta, \alpha)$ is a cocommutative Hom-coassociative coalgebra, {\it i.e.} 
 $\Delta:A\to A\otimes A$ is a linear map and $\alpha: A \to A$ is a linear self map such that
 \begin{eqnarray}
\Delta&=&\tau\circ\Delta  \quad \mbox{ \rm (cocommutativity), } \label{eq:cocommutativity-Delta}\\
\Delta\circ\alpha&=&\alpha^{\otimes2}\circ\Delta  \quad \mbox{ \rm (comultiplicativity), } \label{hp1p}\\
 (\alpha\otimes\Delta)\circ\Delta&=&(\Delta\otimes\alpha)\circ\Delta \quad  \mbox{ \rm (Hom-coassociativity); }\label{hp2p}
 \end{eqnarray}
\item[ii)] $(A, \gamma, \alpha)$ is  a Hom-Lie coalgebra, {\it i.e.} $\gamma:A\to A\otimes A$ 
and $\alpha:A\to A$ are linear maps such that
\begin{eqnarray}
  \gamma&=&-\tau\circ\gamma \quad \mbox{\rm (skew-cosymmetry)},\label{p31}\\
\gamma\circ\alpha&=&\alpha^{\otimes2}\circ\gamma \quad\mbox{\rm (comultiplicativity)},\label{p41}\\
(Id_A+\sigma+\sigma^2)\circ(\alpha\otimes\gamma)\circ\gamma&=&0 \quad \mbox{\rm (Hom-coJacobi identity)};\label{p51}
 \end{eqnarray}
\item [iii)] the Hom-coLeibniz identity holds, {i.e.}
\begin{eqnarray}
 (\alpha\otimes\Delta)\circ\gamma = (\gamma\otimes\alpha)\circ\Delta 
 + (\tau\otimes Id_A)\circ(\alpha\otimes\gamma)\circ\Delta.\label{p61}
\end{eqnarray}
\end{enumerate}
\end{definition}
\begin{remark}
With the notations of Definition \ref{hpdd}, we have the following.
\begin{enumerate}
\item  If $\Delta$ is non-cocommutative, that is Equation (\ref{eq:cocommutativity-Delta})  is not satisfied, 
then $(A, \Delta, \gamma, \alpha)$ is said to be a non-cocommutative Hom-Poisson coalgebra.
\item When $\alpha=Id_A$, $(A, \Delta, \gamma,  Id_A)$, simply denoted $(A, \Delta, \gamma)$, 
is a Poisson coalgebra.
\item In the Sweedler's notation, the conditions (\ref{hp1p}),  (\ref{hp2p}) and (\ref{p61}) 
mean respectively that 
\end{enumerate}
\begin{eqnarray}
 \sum\alpha(x)_1\otimes\alpha(x)_2&=&\sum\alpha(x_1)\otimes\alpha(x_2), \label{hp1}\\
\sum \alpha(x_1)\otimes x_{21}\otimes x_{22}&=&\sum x_{11}\otimes x_{12}\otimes\alpha(x_2), \label{hp2}\\
\sum\alpha(x_{(1)})\otimes x_{(2)1}\otimes x_{(2)2}&=&
(x_{1})_{(1)}\otimes (x_1)_{(2)}\otimes\alpha(x_2)  %\nonumber\\
+(x_2)_{(1)}\otimes\alpha(x_1)\otimes(x_2)_{(2)}, \label{hp2a}
\end{eqnarray}
where $\Delta(x)=\sum x_1\otimes x_2$ and $\gamma(x)=\sum x_{(1)}\otimes x_{(2)}$.
%\end{enumerate}
\end{remark}
% {\color{red} D\'efinir {\bf morphism of coassociative coalgebras } et {\bf morphism of Lie coalgebras} }
\begin{definition}
 Let $(A,\Delta, \alpha)$ and $(A',\Delta', \alpha')$ be two Hom-coassociative algebras. 
A linear map $f : A\rightarrow A$ is said to be a morphism of Hom-coassociative coalgebras if 
$$f\circ\alpha=\alpha'\circ f\quad \mbox{and}\quad \Delta\circ f=f^{\otimes2}\circ\Delta.$$
\end{definition}
We have a similar definition for morphisms of Hom-Lie coalgebras.
% Observe that any twisting map of Hom-poisson coalgebras are both morphism of Hom-coassociative coalgebras and morphism
% of Hom-Lie coalgebras.
\begin{definition}
If  $(A,\Delta,\gamma,\alpha)$ and $(A',\Delta',\gamma',\alpha')$ are Hom-Poisson coalgebras 
(not necessarily cocommutative), a linear map $f:A\to  A'$ is a morphism of Hom-Poisson coalgebras 
if it is a morphism of the underlying coassociative coalgebras and Lie coalgebras 
such that $f\circ\alpha=\alpha'\circ f$.
\end{definition}
% {\color{red} Le lemme suivant est-il original, i.e. est-il un r\'esultat nouveau ?} oui
\begin{lemma}\label{+}
Suppose $(A, \Delta, \gamma, \alpha)$ is a non-cocommutative Hom-Poisson coalgebra, then 
$(A, \Delta^{op}, \gamma, \alpha)$ and $(A, -\Delta, -\gamma, \alpha)$
 are also non-cocommutative Hom-Poisson coalgebras.
\end{lemma}
%{\color{red}
\begin{proof} 
The fact that  $(A, -\Delta, -\gamma, \alpha)$ is a non-cocommutative Hom-Poisson coalgebras 
is straightforward and   follows easily from Definition \ref{hpdd}. 
Now, let us prove that $(A, \Delta^{op}, \gamma, \alpha)$ is also  a  non-cocommutative Hom-Poisson 
coalgebra.  By hypothesis, $(A,\gamma, \alpha)$ is a Hom-Lie coalgebra. 
Thus we have to prove the relations (\ref{hp1p}), (\ref{hp2p}) and (\ref{p61}) for $\Delta^{op}$.  
Let  $x$ be in $A$.

\noindent
1) Comultiplicativity of $\Delta^{op}$:
\begin{eqnarray}
(\Delta^{op}\circ\alpha)(x)&=&\Delta^{op}(\alpha(x)) \cr 
&=&(\alpha(x))_2\otimes(\alpha(x))_1\nonumber\\
&=&\alpha(x_2)\otimes\alpha(x_1)\text{ (by (\ref{hp1})) }\nonumber\\
&=&\alpha^{\otimes2}(x_2\otimes x_1) \cr 
&=&\alpha^{\otimes2}(\Delta^{op}(x))\nonumber\\
&=&(\alpha^{\otimes2}\circ\Delta^{op})(x)\nonumber.
\end{eqnarray}
\noindent
2) Hom-coassociativity of $\Delta^{op}$: 
\begin{eqnarray}
(\alpha\otimes\Delta^{op})\circ\Delta^{op}(x) &=&(\alpha\otimes\Delta^{op})(\tau\circ\Delta(x)) \cr 
&=&(\alpha\otimes\Delta^{op})(x_2\otimes x_1) \cr 
&=&\alpha(x_2)\otimes \Delta^{op}(x_1) \cr 
&=&\alpha(x_2)\otimes \tau(\Delta(x_1))\cr 
&=&\alpha(x_2)\otimes x_{12}\otimes x_{11}\cr 
&=&x_{22}\otimes x_{21}\otimes \alpha(x_1) \text{ (by (\ref{hp2}))}\nonumber\\
&=&\tau(x_{21}\otimes x_{22})\otimes \alpha(x_1) \cr 
&=& (\tau\circ\Delta)(x_2)\otimes \alpha(x_1) \cr 
&=&\Delta^{op}(x_2)\otimes \alpha(x_1)\nonumber\\
&=&(\Delta^{op}(x_2)\otimes\alpha(x_1)) \cr 
&=&(\Delta^{op}\otimes\alpha)(x_2\otimes x_1) \cr 
&=&(\Delta^{op}\otimes\alpha)(\tau\circ\Delta(x)) \cr 
&=&(\Delta^{op}\otimes\alpha)\circ\Delta^{op}(x).\nonumber
\end{eqnarray}
\noindent
3) Hom-coLeibniz identity:
\begin{eqnarray}
 (\alpha\otimes\Delta^{op})\circ\gamma(x)\!\!&=&\alpha(x_{(1)})\otimes x_{(2)2}\otimes x_{(2)1}\nonumber\\
&=&(Id_A\otimes\tau)\Big(\alpha(x_{(1)})\otimes x_{(2)1}\otimes x_{(2)2}\Big)\nonumber\\
&=&(Id_A\otimes\tau)\Big((x_1)_{(1)}\otimes (x_1)_{(2)}\otimes\alpha(x_2)%\nonumber\\
%&&
+(x_2)_{(1)}\otimes\alpha(x_1)\otimes(x_2)_{(2)}\Big) \mbox{ \small (by (\ref{hp2a})) }\nonumber\\
&=&(x_1)_{(1)}\otimes\alpha(x_2)\otimes (x_1)_{(2)}+(x_2)_{(1)}\otimes(x_2)_{(2)}\otimes\alpha(x_1)\nonumber\\
&=&(\tau\otimes Id_A)\Big(\alpha(x_2)\otimes(x_1)_{(1)}\otimes(x_1)_{(2)}\Big)+\gamma(x_2)\otimes\alpha(x_1)\nonumber\\
&=&(\tau\otimes Id_A)\circ(\alpha\otimes\gamma)(x_2\otimes x_1)+(\gamma\otimes\alpha)(x_2\otimes x_1)\nonumber\\
&=&(\tau\otimes Id_A)\circ(\alpha\otimes\gamma)\circ\Delta^{op}(x)+(\gamma\otimes\alpha)\circ\Delta^{op}(x)\nonumber. 
\end{eqnarray}
Thus $(A, \Delta^{op}, \gamma, \alpha)$ is a non-cocommutative Hom-Poisson coalgebra. 
\end{proof}
%}
The Lemma below (see also \cite[Corollary 3.3]{bordeman-elchinger-makhlouf:twisting-poisson-copoisson-quantization}) 
allows to get a Hom-Poisson coalgebra
 (not necessarily cocommutative)  from a Poisson coalgebra (not necessarily cocommutative) via a Poisson coalgebra endomorphism.
 Let $(A,\Delta, \gamma)$ be a Poisson coalgebra  and $\alpha: A\to A$ be a Poisson coalgebra
 endomorphism. Set $\Delta_\alpha=\Delta\circ\alpha$ and  $\gamma_\alpha=\gamma\circ\alpha$. 
\begin{lemma}[\cite{bordeman-elchinger-makhlouf:twisting-poisson-copoisson-quantization}]\label{hpatt}
%  Let $(A,\Delta, \gamma)$ be a Poisson coalgebra (not necessarily commutative) and $\alpha: A\to A$ be a Poisson coalgebra
%  endomorphism. Then 
The quadruple $(A,\Delta_\alpha, \gamma_\alpha, \alpha)$  is a Hom-Poisson coalgebra (not necessarily commutative). 
\end{lemma}

\section{Modules over left Hom-alternative algebras \label{sec:modules-alternative-algebras}}
% {\color{red} Faire une petite introduction de la section \'enon\c{c}ant l'objectif}
In this section, we introduce modules over left Hom-alternative algebras. We prove that they are close under twisting. Then
we show that left alternative algebras morphisms extend to morphisms of left Hom-alternative algebras and we give a relationship
between modules over left Hom-alternative algebras and modules over right Hom-alternative algebras.
% {\color{red} N y a-t-il pas lieu de parler de {\bf left Hom-alternative module} et de \\ {\bf right Hom-alternative module} et ensuite 
% de {\bf Hom-alternative module}}. 

% Oh non. Je ne pense pas. Cela me fait plutot penser aux bimodules over left-Hom-alternative algebras.
% Vous pouvez quand meme le definir puisqu'on en parle dans la proposition 3.2 mais c'est pas la peine d'étudier ses propriétés.
\begin{definition}\label{mhla}
Let $(A, \mu,\alpha)$ be a left (resp. right) Hom-alternative algebra  and  $(M, \beta)$ be a Hom-module. 
An $A$-module structure on $M$ consists of a {\bf K}-bilinear map
 $\mu_M :A\otimes M\to M$ (resp. $\mu_M : M\otimes A\rightarrow M$) such that for any $x, y\in A$ and every $m\in M$,
\begin{eqnarray}
 \mu_M\Big(\alpha(x), \mu_M(x, m)\Big)\!=\!\mu_M\Big(\mu(x, x), \beta(m)\Big), \text{ resp. } 
 \mu_M\Big(\mu_M(m, x), \alpha(x)\Big)\!=\! \mu_M\Big(\beta(m), \mu(x, x)\Big).  \label{hai1}
\end{eqnarray}
% respectively,
% \begin{eqnarray}
% \mu_M\Big(\mu_M(m, x), \alpha(x)\Big)= \mu_M\Big(\beta(m), \mu(x, x)\Big). \label{hai11}
% \end{eqnarray}
If $M$ carries a $A$-module structure, then it is called a module over $A$ or an $A$-module.
\end{definition}
\begin{example} {\normalfont
Any left (resp. right) Hom-alternative algebra is a module over itself.
}
\end{example}
\begin{example}{\normalfont
 Let $(A, \mu,\alpha)$ be a multiplicative left Hom-alternative algebra. Then $(A, \alpha)$ is an $A$-module with the action
$x\ast m=\mu(\alpha(x), m)$ for any $x, m\in A$.
In fact, we have
\begin{eqnarray}
 \alpha(x)\ast(x\ast m)&=&\alpha(x)\ast\mu(\alpha(x),m) \cr 
                       &=&\mu\Big(\alpha^2(x), \mu\big(\alpha(x), m\big)\Big)\nonumber\\
&=&\mu\Big(\mu\big(\alpha(x), \alpha(x)\big), \alpha(m)\Big)\cr 
&=&\mu\Big(\alpha\big(\mu(x, x)\big), \alpha(m)\Big)\nonumber\\
&=&\mu(x, x)\ast\alpha(m).\nonumber
\end{eqnarray}}
\end{example}
\begin{definition}
 Let $(M, \mu_M, \beta)$ and $(M', \mu_{M'}, \beta')$ be two modules over a left (resp. right) Hom-alternative algebra 
 $(A, \mu, \alpha)$. A linear map $f : M\rightarrow M'$ is said to be a morphism of left (resp. right) Hom-alternative modules if, 
 for any $x\in A$ and every $m\in M$,
 \begin{equation}
  f\Big(\mu_M(x, m)\Big)=\mu_{M'}\Big(x, f(m)\Big) \quad  \Big[\text{resp. }  
  f\Big(\mu_{M}(m, x)\Big)=\mu_{M'}\Big(f(m), x\Big)\Big].
 \end{equation}
\end{definition}
\begin{example}{\normalfont
 Any morphism of left (resp. right) Hom-alternative algebras is a morphism of left (resp. right) Hom-alternative modules.}
\end{example}
\begin{definition}
Let $(M, \mu_M, \beta)$ be a Module over a left Hom-alternative algebra $(A, \mu,\alpha)$.  
We call {\it module Hom-associator} associated to the
module $M$, the trilinear map $\hbox{\rm as}(\alpha, \beta)$, defined by: for all $x,y \in A$ and $m\in M$, 
\begin{equation}
\hbox{\rm as}(\alpha, \beta)(x, y, m)=\mu_M\Big(\alpha(x), \mu_M(y, m)\Big)-\mu_M\Big(\mu(x, y), \beta(m)\Big). \label{hai3}
\end{equation}
\end{definition}
Using the {\it module Hom-associator}, the condition (\ref{hai1}) may be written  as: %$\forall x\in A, m\in M$,
\begin{equation}
\hbox{\rm as}(\alpha, \beta)(x, x, m)= 0. \label{hai4}
\end{equation}
\begin{remark}
The Hom-associator $\hbox{\rm as}_\alpha$, associated to a Hom-associative algebra $(A, \mu, \alpha)$
 is obtained by taking  $\beta=\alpha$ in (\ref{hai3}).
\end{remark}
%Now, let us fixe some notations. 
\begin{theorem}
 Let $(A, \mu,\alpha)$ be a left Hom-alternative algebra;   $(M, \mu_M, \beta_M)$ and $(N, \mu_N, \beta_N)$ be two $A$-modules.
Define the bilinear map $\rho : A\otimes M\otimes N\rightarrow M\otimes N$ and the linear map $\beta : M\otimes N\rightarrow M\otimes N$ as
\begin{equation}
\rho(x\otimes m\otimes n)=\mu_M(x, m)\otimes \mu_N(x, n) \; \mbox{ and } \;  \beta(m\otimes n)=\beta_M(m)\otimes\beta_N(n). 
\end{equation}
Then $(M\otimes N, \rho, \beta)$ is an $A$-module.
\end{theorem}
\begin{proof}
 For any $x, y\in A$ and every $ m, n\in M$, one has:
\begin{eqnarray}
 \rho\circ(\mu\otimes\beta)(x\otimes y\otimes m\otimes n)
&=&\rho\Big(\mu(x, y)\otimes\beta_M(m)\otimes\beta_N(n)\Big)\nonumber\\
&=&\mu_M\Big(\mu(x, y),\beta_M(m)\Big)\otimes\mu_N\Big(\mu(x, y),\beta_N(n)\Big)\nonumber\\
&=&\mu_M\Big(\alpha(x), \mu_M(y, m)\Big)\otimes\mu_N\Big(\alpha(x), \mu_N(y, n)\Big)\nonumber\\
&=&\rho\Big(\alpha(x)\otimes\mu_M(y, m)\otimes\mu_N(y, n)\Big)\nonumber\\
&=&\rho\Big(\alpha(x)\otimes \rho(y\otimes m\otimes n)\Big)\nonumber\\
&=&\rho\circ(\alpha\otimes\rho)(x\otimes y\otimes m\otimes n).\nonumber
\end{eqnarray}
This ends the proof.
\end{proof}
For the rest of this section, $(M,\mu_M,\beta)$ is a module over a left Hom-alternative algebra
$(A,\mu,\alpha)$ and  we set $\tilde\mu_M=\mu_M\circ(\alpha^2\otimes Id_M)$. 
\begin{proposition} 
 %Let $(M, \mu_M, \beta)$ be a module over the multiplicative left Hom-alternative algebra $(A, \mu, \alpha)$. Then
The triple $(M, \tilde\mu_M, \beta)$ is another module over $A$.
%  the left Hom-alternative algebra $(A, \mu, \alpha)$.
\end{proposition}
\begin{proof}
Using $\mu_M\Big(\alpha(x), \mu_M(x, m)\Big)=\mu_M\Big(\mu(x, x), \beta(m)\Big)$, one has
\begin{eqnarray}
 \tilde\mu_M\Big(\alpha(x), \tilde\mu_M(x, m)\Big) &=& \tilde\mu_M\Big(\alpha(x), \mu_M(\alpha^2(x), m)\Big)\nonumber\\
                                                   &=& \mu_M\Big(\alpha^3(x),\mu_M\big(\alpha^2(x), m\big)\Big)\nonumber\\
                                                   &=& \mu_M\Big(\mu\big(\alpha^2(x), \alpha^2(x)\big), \beta(m)\Big) \nonumber\\
                                                   &=& \mu_M\Big(\alpha^2\mu(x, x), \beta(m)\Big) \nonumber\\ 
                                                   &=& \tilde\mu_M\Big(\mu(x, x), \beta(m)\Big)\nonumber.
\end{eqnarray}
This completes the proof.
\end{proof}
% \begin{corollary}
%  g
% \end{corollary}

% \begin{corollary}
%  Let $(A,\alpha\circ \mu,\alpha)$ be a left Hom-alternative algebra (as in Lemma \ref{lhatl}) and $(M, \mu_M, \beta)$ 
%  a module over $(A,\alpha\circ \mu,\alpha)$.  Then $(M, \tilde\mu_M, \beta)$ is also a module over $(A,\alpha\circ \mu,\alpha)$.
% \end{corollary}
% \begin{proof}
%  .................................................
%  
%  \vskip 3cm 
%  
% \end{proof}
%\begin{remark} 
 We have a similar formulation for modules over right Hom-alternative algebras.
%\end{remark}
\begin{theorem}\label{morfi}
Let $f : (M, \mu_M, \beta)\rightarrow (M', \mu_{M'}, \beta')$ be a morphism of modules over a left Hom-alternative algebra
$(A, \mu, \alpha)$. Then, $f : (M, \tilde\mu_M, \beta)\rightarrow (M', \tilde\mu_{M'}, \beta')$ %, 
% where $\tilde\mu_M=\mu_M\circ(\alpha^2\otimes Id_M)$,
%  $\tilde\mu_{M'}=\mu_{M'}\circ(\alpha^2\otimes Id_M)$,
is a morphism of modules over the left Hom-alternative algebra $(A, \mu, \alpha)$.
\end{theorem}
\begin{proof}
Using the equality $f\circ\mu_M=\mu_{M'}\circ(Id_A\otimes f)$, one has for any $x\in A$ and all $m\in M$,
$
 f\Big(\tilde\mu_M(x, m)\Big)=f\Big(\mu_M(\alpha^2(x), m)\Big)=\mu_{M'}\Big(\alpha^2(x), f(m)\Big)=\tilde\mu_{M'}\Big(x, f(m)\Big)
$.
\end{proof}
% \begin{corollary}
% Let $f : (M, \mu_M, \beta)\rightarrow (M', \mu_{M'}, \beta')$ be a morphism of modules over a left alternative algebra $(A, \mu, Id_A)$ and 
% $\alpha : A\rightarrow A$ an algebra endomorphism.
% Then, $f : (M, \tilde\mu_M, \beta)\rightarrow (M', \tilde\mu_{M'}, \beta')$
% is a morphism of modules over the left Hom-alternative algebra $(A,\alpha\circ \mu,\alpha)$.
% \end{corollary}
% \begin{proof}
%  .................................;
%  
%  \vskip 3cm 
%  
% \end{proof}
% {\color{red} La Proposition \ref{pro:cns-module-left-hom-alternative} ne devrait-elle pas donner une condition n\'ecessaire et suffisante
% pour le triplet $(A,\mu_M,\beta)$ d'\^etre module sur l'alg\`ebre hom-alternative \`a gauche $(A,\mu,\alpha)$ 
% (analogie avec Proposition \ref{pro:cns-hom-alternative-algebra}) ?}
% Je ne crois pas encore à ça!!!!!!
\begin{proposition}\label{pro:cns-module-left-hom-alternative}
%Let $(M, \mu_M, \beta)$ be a module over the left Hom-alternative algebra $(A, \mu,\alpha)$. Then, 
We have the following identity: 
for any  $x, y\in A$ and any $m\in M$,
\begin{eqnarray}
\mu_M\Big(\alpha(x), \mu_M(y, m)\Big)-\mu_M\Big(\mu(x, y), \beta(m)\Big) +\mu_M\Big(\alpha(y), \mu_M(x, m)\Big)
-\mu_M\Big(\mu(y, x), \beta(m)\Big)=0. \label{had}
\end{eqnarray}
\end{proposition}
\begin{proof}
By assumption $(M, \mu_M, \beta)$ is a left Hom-alternative module {\it i.e.}
$\hbox{\rm as}(\alpha, \beta)(x, x, m)=0$, for any $x \in A$ and $m\in M$. The relation (\ref{had})
 follows by expanding $\hbox{\rm as}(\alpha, \beta)(x+y, x+y, m)=0$, for any $x, y\in A$ and any  $m\in M$.
\end{proof}
\begin{remark}
  The condition (\ref{had}) is equivalent to $\hbox{\rm as}(\alpha, \beta)(x, y, m)+\hbox{\rm as}(\alpha, \beta)(y, x, m)=0$, 
  for all $x,y \in A$ and $m\in M$. Now, taking
$\hbox{\rm as}(\alpha, \beta)(x, y, m)-\hbox{\rm as}(\alpha, \beta)(y, x, m)=0$, we obtain the relation 
connecting $A$-modules and $L(A)$-modules (\cite{BI}), 
where $L(A)$ is the Hom-Lie algebra associated to the Hom-associative algebra $A$.
\end{remark}
Here is an extension of Lemma \ref{-t}.
 \begin{proposition}
%Let $(M, \mu_M, \beta)$ be a module over the left Hom-alternative algebra $(A, \mu,\alpha)$. Then,
% \begin{enumerate}
%  \item
 $(M,\!-\mu_M, \beta)$ is  a module over the left Hom-alternative algebra $(A,\!-\mu,\alpha)$.
% \item $(M, \mu_M^{op}, \beta)$ is a module over the right Hom-alternative algebra $(A, \mu^{op},\alpha)$.
% \end{enumerate}
\end{proposition}
\begin{proof}
 The proof is straightforward and comes from Lemma \ref{-t} and Definition \ref{mhla}.
\end{proof}

\section{Hom-Poisson comodules \label{sec:hom-poisson-comodules}}

In this section, we construct some Poisson comodules and prove that twisting the module structure map of a Hom-Poisson comodule 
we get another one.

\subsection{Hom-(coassociative) Lie comodules}

\begin{definition}[\cite{zhang:comodule-hom-coalgebras}]
Let $(C, \delta, \alpha)$ be a Hom-coassociative coalgebra  and  $(M, \beta)$ be a Hom-module. 
 A $C$-comodule structure on $M$ consists of a linear map (called structure map) 
 $\Delta: M\longrightarrow C\otimes M$, $m\mapsto\sum m_{(-1)}\otimes m_{(0)}$ such that
$\Delta\circ\beta=(\alpha\otimes\beta)\circ\Delta$  and 
$(\alpha\otimes\Delta)\circ\Delta=(\delta\otimes \beta)\circ\Delta$. 
% \begin{eqnarray}
% \Delta\circ\beta&=&(\alpha\otimes\beta)\circ\Delta, \label{moascds}\\
%  (\alpha\otimes\Delta)\circ\Delta&=&(\delta\otimes \beta)\circ\Delta. \label{ami}
% \end{eqnarray}
\end{definition}
% 
% \begin{definition}[\cite{zhang:comodule-hom-coalgebras}]
% Let $(C, \delta, \alpha)$ be a Hom-coassociative coalgebra  and  $(M, \beta)$ be a Hom-module. 
%  A $C$-comodule structure on $M$ consists of a linear map $\Delta: M\to C\otimes M$, $m\mapsto\sum m_{(-1)}\otimes m_{(0)}$ such that
% {\color{blue}
% \begin{eqnarray}
% (\alpha\otimes\Delta)\circ \Delta=(\delta\otimes\beta)\circ \Delta.
% \end{eqnarray}
% }
% {\color{red}
% ET NON  
% \begin{eqnarray}
% \delta\circ\beta&=&(\alpha\otimes\beta)\circ\delta \label{moascds}\\
%  (\alpha\otimes\delta)\circ\delta&=&(\delta\otimes \beta)\circ\delta. \label{ami}
% \end{eqnarray}
% }
% \end{definition}
% {\color{red} Cette d\'efinition m'a l'air inexacte car $\Delta$ ne joue aucun r\^ole. Il faut, me semble-t-il, 
% enlever les deux \'equation en rouge. Voir \cite{zhang:comodule-hom-coalgebras}}
\begin{lemma}[\cite{zhang:comodule-hom-coalgebras}]\label{smc}
 Consider  a Hom-coassociative coalgebra $(C, \delta, \alpha)$  and a $C$-comodule $(M,\beta)$  with structure map
 $\Delta: M\rightarrow C\otimes M$.  Then, the map  $\tilde\Delta=(\alpha^2\otimes Id_M)\circ\Delta : M\to C\otimes M$ is  a structure map 
 of another $C$-comodule structure on $M$.
\end{lemma}
The following definition is the Hom-type of the one given in \cite{ZLY}.
\begin{definition}[\cite{BI}]
 Let $(L, \gamma, \alpha)$ be a Hom-Lie coalgebra and $(M, \beta)$ be a Hom-module. If there exists a linear map 
$\Gamma : M\to L\otimes M$ such that 
\begin{eqnarray}
\Gamma\circ\beta=(\alpha\otimes\beta)\circ\Gamma \text{ and } 
 (\gamma\otimes \beta)\circ\Gamma=(\alpha\otimes \Gamma)\circ\Gamma
 -(\tau\otimes Id_M)\circ(\alpha\otimes\Gamma)\circ\Gamma, \label{mod}
\end{eqnarray}
% \begin{eqnarray}
% \Gamma\circ\beta&=&(\alpha\otimes\beta)\circ\Gamma, \label{modh}\\
%  (\gamma\otimes \beta)\circ\Gamma&=&(\alpha\otimes \Gamma)\circ\Gamma-(\tau\otimes Id_L)\circ(\alpha\otimes\Gamma)\circ\Gamma, \label{mod}
% \end{eqnarray}
then $M$ is called a $L$-comodule. %{\color{red}$Id_M$ ou $Id_L$}
\end{definition}
\begin{remark} \mbox{}\\
1) If $\alpha=Id_L$ and $\beta=Id_M$, we recover the definition of Lie comodules (\cite{ZLY}).\\
2) The conditions  (\ref{mod}) may be respectively written as
\begin{eqnarray}
 (\beta(v))_{(-1)}\otimes(\beta(v))_{(0)}\!\!\!\!&=&\!\!\!\!\alpha(v_{(-1)})\otimes \beta(v_{(0)}), \label{expm1}\\
v_{(-1)(1)}\otimes v_{(-1)(2)}\otimes\beta(v_{(0)})
\!\!\!\!&=&\!\!\!\!\alpha(v_{(-1)})\otimes v_{(0)(-1)}\otimes v_{(0)(0)}-v_{(0)(-1)} \otimes\alpha(v_{(-1)})\otimes v_{(0)(0)},\label{expm2}
\end{eqnarray}
where  $\gamma(x)=x_{(1)}\otimes x_{(2)}$ and $\Gamma(v)=v_{(-1)}\otimes v_{(0)}$ (without summation symbol, for symplicity).
\end{remark}
The following statement is the Lie-type of Lemma \ref{smc}.
\begin{lemma}\label{csml}
 Let $(L, \gamma,\alpha)$ be a Hom-Lie coalgebra and $M$ be a $L$-comodule with structure map
 $\Gamma : M\longrightarrow L\otimes M$. Define the map 
$\tilde\Gamma=(\alpha^2\otimes Id_M)\circ\Gamma : M \to L\otimes M$.
Then $\tilde\Gamma$ is a structure map of another $L$-comodule structure on $M$.
\end{lemma}
\begin{proof}
For any $m\in M$, we have on one hand:
\begin{eqnarray}
 (\tilde\Gamma\circ\beta)(m) &=&(\alpha^2\otimes Id_M)\circ\Gamma(\beta(m))\nonumber\\
                             &=&(\alpha^2\otimes Id_M)\circ(\alpha\otimes\beta)\circ\Gamma(m)\cr 
                             &=&(\alpha^3\otimes\beta)\circ\Gamma(m)\nonumber\\
                             &=&(\alpha\otimes\beta)\circ(\alpha^2\otimes Id_M)\circ\Gamma(m)\cr 
                             &=&(\alpha\otimes\beta)\circ\tilde\Gamma(m);\nonumber
\end{eqnarray}
and on the other hand: %{\color{red} $Id_M$ ou $Id_L$}
\begin{eqnarray}
(\gamma\otimes\beta)\circ\tilde\Gamma(m)\!\!\!\!&=&\!\!\!\!(\gamma\otimes\beta)\circ(\alpha^2\otimes Id_M)\circ\Gamma(m) \cr 
&=&\!\!\!\!(\gamma\circ\alpha^2\otimes\beta)\circ\Gamma(m)\cr 
&=&\!\!\!\!\Big((\alpha^2)^{\otimes2}\otimes Id_M\Big)\circ(\gamma\otimes\beta)\circ\Gamma(m) \text{ (using (\ref{p41}))}\nonumber\\
&=&\!\!\!\! \Big((\alpha^2)^{\otimes2}\otimes Id_M\Big)\Big[(\alpha\otimes\Gamma)\circ\Gamma(m)
   -(\tau\otimes Id_M)\circ(\alpha\otimes\Gamma)\circ\Gamma(m)\Big]\nonumber\\
&=&\!\!\!\!\Big((\alpha^2)^{\otimes2}\!\otimes\! Id_M\Big)\Big[(\alpha\!\otimes\!\Gamma)(m_{(-1)}\!\otimes\! m_{(0)})
   \!-\!(\tau\!\otimes\! Id_M)(\alpha(m_{(-1)})\!\otimes\!\Gamma(m_{(0)}))\Big]\nonumber\\
&=&\!\!\!\!\Big((\alpha^2)^{\otimes2}\!\otimes\! Id_M\Big)\Big[\alpha(m_{(-1)})\!\otimes\! m_{(0)(-1)}\!\otimes\! m_{(0)(0)}
   \!-\! m_{(0)(-1)}\!\otimes\! \alpha(m_{(-1)})\!\otimes\! m_{(0)(0)}\Big]\nonumber\\
&=&\!\!\!\!\alpha^3(m_{(-1)})\otimes\alpha^2(m_{(0)(-1)})\otimes m_{(0)(0)}
   -\alpha^2(m_{(0)(-1)})\otimes \alpha^3(m_{(-1)})\otimes m_{(0)(0)}\nonumber\\
&=&\!\!\!\!\alpha^3(m_{(-1)})\!\otimes\!\Big[(\alpha^2\!\otimes\! Id_M)\circ\Gamma(m_{(0)})\Big]
   \!-\!(\tau\!\otimes\! Id_M)\Big(\alpha^3(m_{(-1)})\!\otimes\! (\alpha^2\!\otimes\! Id_M)\Gamma(m_{(0)})\Big)\nonumber\\
&=&\!\!\!\!\alpha^3(m_{(-1)})\otimes\tilde\Gamma(m_{(0)})
    -(\tau\otimes Id_M)\Big(\alpha^3(m_{(-1)})\otimes\tilde\Gamma_M(m_{(0)})\Big)\nonumber\\
&=&\!\!\!\!(\alpha\otimes\tilde\Gamma)\Big(\alpha^2(m_{(-1)})\otimes m_{(0)}\Big)
-(\tau\otimes Id_M)\circ(\alpha\otimes\tilde\Gamma)\Big(\alpha^2(m_{(-1)})\otimes m_{(0)}\Big)\nonumber\\
&=&\!\!\!\!(\alpha\otimes\tilde\Gamma)\circ(\alpha^2\otimes Id_M)\circ\Gamma(m)
-(\tau\otimes Id_M)\circ(\alpha\otimes\tilde\Gamma)\circ(\alpha^2\otimes Id_M)\circ\Gamma(m)\nonumber\\
&=&\!\!\!\!(\alpha\otimes\tilde\Gamma)\circ\tilde\Gamma(m)-(\tau\otimes Id_M)\circ(\alpha\otimes\tilde\Gamma)\circ\tilde\Gamma(m)\nonumber.
\end{eqnarray}
Therefore $(M,\tilde\Gamma,\beta)$ is another $L$-comodule.
\end{proof}
We have the following result which is the Hom-analogue of Lemma 1 in \cite{ZLY}.
\begin{proposition}
Let $(V,\beta_V)$ and $(W,\beta_W)$ be two modules over the Hom-Lie coalgebra $(L, \gamma, \alpha)$ with structure maps 
$\Gamma_V:v\longmapsto\sum v_{(-1)}\otimes v_{(0)}$ and $\Gamma_W: w \longmapsto\sum w_{(-1)}\otimes w_{(0)}$, respectively.
If we set $M=V\otimes W$, $\beta=\beta_V\otimes\beta_W$ and
\begin{equation}
 \Gamma_M : M\otimes M, \quad v\otimes w \mapsto \sum[\alpha(v_{(-1)})\otimes v_{(0)}\otimes\beta_W(w)
+\alpha(w_{(-1)})\otimes\beta_V(v)\otimes w_{(0)}],\nonumber
\end{equation}
then  $(M,\beta)$ is a comodule over $L$.
\end{proposition}
\begin{proof}
We must prove (\ref{mod}) for $\Gamma_M$ and $\beta$. We have, for any $v\in V$ and all $w\in W$,
\begin{eqnarray}
\Gamma_M\!\circ\!\beta(v\otimes w)\!\!\!\!\!&=&\!\!\!\! \Gamma_M\Big(\beta_V(v)\otimes\beta_W(w)\Big)\nonumber\\
&=&\!\!\!\!\alpha\Big((\beta_V(v))_{(-1)}\Big)\otimes\Big(\beta_V(v)\Big)_{(0)}\otimes\beta_W^2(w)+
\alpha\Big(\Big(\beta_W(w)\Big)_{(-1)}\Big)\otimes\beta_V^2(v)\otimes\Big(\beta_W(w)\Big)_{(0)}\nonumber\\
&=&\!\!\!\!\!\alpha\Big(\!(\beta_V(v)\!)_{(-1)}\!\Big)\!\otimes\!\Big(\beta_V(v)\Big)_{(0)}\!\otimes\!\beta_W^2(w) 
\! +\!(\tau\!\otimes\! Id)\Big[\beta_V^2(v)\!\otimes\!\alpha\Big(\!\Big(\beta_W(w)\!\Big)_{(-1)}\!\Big)\!\otimes\!\Big(\beta_W(w)\Big)_{(0)}\!\Big]\nonumber\\
&=&\!\!\!\!\alpha^2\Big(v_{(-1)}\Big)\!\otimes\!\beta_V\Big(v_{(0)}\Big)\!\otimes\!\beta_W^2(w)
\!+\!(\tau\!\otimes\! Id)\Big[\beta_V^2(v)\!\otimes\!\alpha^2\Big(w_{(-1)}\Big)\!\otimes\!\beta_W\Big(w_{(0)}\Big)\Big]\mbox{ (by (\ref{expm1})) }\nonumber\\
%&=&\alpha^2\Big(v_{(-1)}\Big)\otimes\beta_V(v_)\otimes\beta_W^2(w)+\alpha^2(w_)\otimes\beta_V^2(v)\otimes\beta_W(w_)\nonumber\\
&=&\!\!\!\!(\alpha\otimes\beta_V\otimes\beta_W)\Big[\alpha\Big(v_{(-1)}\Big)\otimes v_{(0)}\otimes\beta_W(w)
+\alpha\Big(w_{(-1)}\Big)\otimes\beta_V(v)\otimes w_{(0)}\Big]\nonumber\\
&=&\!\!\!\!(\alpha\otimes\beta)\circ\Gamma_M(v\otimes w).\nonumber
\end{eqnarray}
This prove the first equality of (\ref{mod}). Next,  we set $L=(\alpha\otimes\Gamma_M)\circ\Gamma_M(v\otimes w)$. Then, 
\begin{eqnarray}
L\!\!\!\!\!&=&\!\!\!\!\!(\alpha\otimes\Gamma_M)\Big(\alpha\Big(v_{(-1)}\Big)\otimes v_{(0)}\otimes\beta_W(w)
+\alpha\Big(w_{(-1)}\Big)\otimes\beta_V(v)\otimes w_{(0)}\Big)\nonumber\\
&=&\!\!\!\!\!\alpha^2\Big(v_{(-1)}\Big)\otimes\Gamma_M\Big(v_{(0)}\otimes\beta_W(w)\Big)
+\alpha^2\Big(w_{(-1)}\Big)\otimes\Gamma_M\Big(\beta_V(v)\otimes w_{(0)}\Big)\nonumber\\
&=&\!\!\!\!\!\alpha^2\Big(v_{(-1)}\Big)\!\otimes\!\alpha\Big(v_{(0)(-1)}\Big)\!\otimes\! v_{(0)(0)}\!\otimes\!\beta_W^2(w)
\!+\!\alpha^2\Big(v_{(-1)}\Big)\!\otimes\!\alpha\Big[\Big(\beta_W(w)\Big)_{(-1)}\Big]\!\otimes\!\beta_V\Big(v_{(0)}\Big)
\!\otimes\!\Big(\beta_W(w)\Big)_{(0)}\nonumber\\
& &\!\!\!\!\!+\!\alpha^2\Big(\!w_{(-1)}\!\Big)\!\otimes\!\alpha\Big[\!\Big(\!\beta_V(v)\!\Big)_{(-1)}\!\Big]\!
   \otimes\!\Big(\!\beta_V(v)\!\Big)_{(0)}\!\otimes\!\beta_W\Big(\!w_{(0)}\!\Big)\!+\!\alpha^2\Big(\!w_{(-1)}\!\Big)\!
   \otimes\!\alpha\Big(\!w_{(0)(-1)}\!\Big)\!\otimes\!\beta_V^2(v)\!\otimes\! w_{(0)(0)}.\label{eq:as}
\end{eqnarray}
To finish we again set $P=(\alpha\otimes\Gamma_M)\circ\Gamma_M(v\otimes w)
-(\tau\otimes Id)\circ(\alpha\otimes\Gamma_M)\circ\Gamma_M(v\otimes w)$. Then, 
\begin{eqnarray}
P%&=&\!\!\!\!\!(\alpha\otimes\Gamma_M)\Gamma_M(v\otimes w) -(\tau\otimes Id)(\alpha\otimes\Gamma_M)\Gamma_M(v\otimes w) \cr 
\!\!\!\!\!\!&=&\!\!\!\!\!\alpha^2\Big(v_{(-1)}\Big)\!\otimes\!\alpha\Big(v_{(0)(-1)}\Big)\!\otimes\! v_{(0)(0)} \!\otimes\! \beta_W^2(w) %\cr 
%& &\!\!\!\!\!\!\! 
\!+\!\alpha^2\Big(v_{(-1)}\Big)\!\otimes\! \alpha\Big[\Big(\beta_W(w)\Big)_{(-1)}\Big]\!\otimes\! \beta_V\Big(v_{(0)}\Big)\!\otimes\! \Big(\beta_W(w)\Big)_{(0)}\nonumber\\
& &\!\!\!\!\!\!\!+\alpha^2\Big(w_{(-1)}\Big)\!\otimes\!\alpha\Big[\Big(\beta_V(v)\Big)_{(-1)}\!\Big]\!\otimes\!\Big(\beta_V(v)\Big)_{(0)}\!\!\otimes\!\beta_W\Big(w_{(0)}\Big) %\cr 
%& &
\!\!
+\!\alpha^2\Big(w_{(-1)}\Big)\!\otimes\!\alpha\Big(w_{(0)(-1)}\Big)\!\otimes\!\beta_V^2(v)\!\otimes\! w_{(0)(0)} \cr 
& &\!\!\!\!\!\!\!-\alpha\Big(v_{(0)(-1)}\Big)\!\otimes\!\alpha^2\Big(v_{(-1)}\Big)\!\otimes\! v_{(0)(0)}\!\otimes\!\beta_W^2(w) %\cr 
%& &\!\!\!\!\!\!\! 
\!-\!\alpha\Big[\Big(\beta_W(w)\Big)_{(-1)}\!\Big]\!\otimes\!\alpha^2\Big(v_{(-1)}\Big)\!\otimes\!\beta_V\Big(v_{(0)}\Big)\!\otimes\!\Big(\beta_W(w)\Big)_{(0)} \cr 
& &\!\!\!\!\!\!\!-\alpha\Big[\Big(\beta_V(v)\Big)_{(-1)}\!\Big]\!\otimes\!\alpha^2\Big(w_{(-1)}\Big)\!\otimes\!\Big(\beta_V(v)\Big)_{(0)}\!\otimes\!\beta_W\Big(w_{(0)}\Big) %\cr 
%& &\!\!\!\!\!\!\!
\!-\!\alpha\Big(w_{(0)(-1)}\Big)\!\otimes\!\alpha^2\Big(w_{(-1)}\Big)\!\otimes\!\beta_V^2(v)\!\otimes\! w_{(0)(0)}   \cr 
&=&\!\!\!\!\!\Big\{ \alpha^2\Big(v_{(-1)}\Big)\otimes\alpha\Big(v_{(0)(-1)}\Big)\otimes v_{(0)(0)} \otimes \beta_W^2(w)  % \cr    
%& &
+\alpha^2\Big(w_{(-1)}\Big)\otimes\alpha\Big(w_{(0)(-1)}\Big)\otimes\beta_V^2(v)\otimes w_{(0)(0)} \cr
& &\!\!\!\!\!\!\!-\alpha\Big(v_{(0)(-1)}\Big)\otimes\alpha^2\Big(v_{(-1)}\Big)\otimes v_{(0)(0)}\otimes\beta_W^2(w) %\cr 
%& &
-\alpha\Big(w_{(0)(-1)}\Big)\otimes\alpha^2\Big(w_{(-1)}\Big)\otimes\beta_V^2(v)\otimes w_{(0)(0)}  \Big\}  \cr  %\mbox{ \small ( by (\ref{bac}))}\nonumber\\
& &+\Big\{ \alpha^2\Big(v_{(-1)}\Big)\otimes \alpha\Big[\Big(\beta_W(w)\Big)_{(-1)}\Big]\otimes \beta_V\Big(v_{(0)}\Big)\otimes \Big(\beta_W(w)\Big)_{(0)}\cr
& &+\alpha^2\Big(w_{(-1)}\Big)\otimes\alpha\Big[\Big(\beta_V(v)\Big)_{(-1)}\Big]\otimes\Big(\beta_V(v)\Big)_{(0)}\otimes\beta_W\Big(w_{(0)}\Big) \cr     
& & -\alpha\Big[\Big(\beta_W(w)\Big)_{(-1)}\Big]\otimes\alpha^2\Big(v_{(-1)}\Big)\otimes\beta_V\Big(v_{(0)}\Big)\otimes\Big(\beta_W(w)\Big)_{(0)} \cr
& &-\alpha\Big[\Big(\beta_V(v)\Big)_{(-1)}\Big]\otimes\alpha^2\Big(w_{(-1)}\Big)\otimes\Big(\beta_V(v)\Big)_{(0)}\otimes\beta_W\Big(w_{(0)}\Big) \Big\} \label{eq:P}%\cr 
\end{eqnarray}
Using (\ref{expm1}) and some computational manipulations, one readily verify that the content of the second accolades  of the right hand 
side of the equality (\ref{eq:P}) is zero. Then, the latter  simply reads : 
\newpage
\begin{eqnarray}
P&=&\!\!\!\!\! \alpha^2\Big(v_{(-1)}\Big)\otimes\alpha\Big(v_{(0)(-1)}\Big)\otimes v_{(0)(0)} \otimes \beta_W^2(w)  
+\alpha^2\Big(w_{(-1)}\Big)\otimes\alpha\Big(w_{(0)(-1)}\Big)\otimes\beta_V^2(v)\otimes w_{(0)(0)} \cr
& &\!\!\!\!\!\!\!-\alpha\Big(v_{(0)(-1)}\Big)\otimes\alpha^2\Big(v_{(-1)}\Big)\otimes v_{(0)(0)}\otimes\beta_W^2(w) 
-\alpha\Big(w_{(0)(-1)}\Big)\otimes\alpha^2\Big(w_{(-1)}\Big)\otimes\beta_V^2(v)\otimes w_{(0)(0)}    \cr  
&=&\!\!\!\!\!(\alpha\!\otimes\!\alpha\!\otimes\! I)(\alpha(v_{(-1)})\!\otimes\! v_{(0)(-1)}\!\otimes\! v_{(0)(0)}
\!-\!v_{(0)(-1)}\!\otimes\!\alpha(v_{(-1)})\!\otimes\! v_{(0)(0)})\!\otimes\!\beta_W^2(w)\nonumber\\
&&\!\!\!\!\!+(\alpha\!\otimes\!\alpha\!\otimes\!\tau)(\alpha(w_{(-1)})\!\otimes\! w_{(0)(-1)}\!\otimes\! w_{(0)(0)}\!\otimes\! \beta_V^2(v)
\!-\!w_{(0)(-1)}\!\otimes\!\alpha(w_{(-1)})\!\otimes\! w_{(0)(0)}\!\otimes\!\beta_V^2(v))\nonumber\\
&=&\!\!\!\!\!\alpha(v_{(-1)(1)})\otimes\alpha((v_{(-1)(2)})\otimes\beta_V(v_{(0)})\otimes\beta_W^2(w) \cr 
& &\!\!\!\!\! +\alpha((w_{(-1)(1)})\otimes\alpha((w_{(-1)(2)})\otimes\beta_V^2(v)\otimes\beta_W(w_{(0)})
\mbox{ (by (\ref{expm2}))}\nonumber\\
&=&\!\!\!\!\!\gamma(\alpha(v_{(-1)}))\!\otimes\!\beta_V(v_{(0)})\!\otimes\!\beta_W^2(w)
\!+\!\gamma(\alpha(w_{(-1)}))\!\otimes\!\beta_V^2(v)\!\otimes\!\beta_W(w_{(0)}) \text{ (by comultiplicativity)}\nonumber\\
&=&\!\!\!\!\!(\gamma\!\otimes\!\beta_V\!\otimes\!\beta_W)(\alpha(v_{(-1)})\!\otimes\! v_{(0)}\!\otimes\!\beta_W(w)
\!+\!\alpha(w_{(-1)})\!\otimes\!\beta_V(v)\!\otimes\! w_{(0)})\nonumber\\
&=&\!\!\!\!\!(\gamma\otimes\beta_V\otimes\beta_W)\Gamma_M(v\otimes w)\nonumber\\
&=&\!\!\!\!\!(\gamma\otimes\beta)\Gamma_M(v\otimes w).\nonumber
\end{eqnarray}
% Le 2em terme de la 1ere colonne s'annule avec le 3em terme de la 2em colonne en utilisant (27) et 29. De meme que pour les 2 autres.\\
Thus (\ref{mod}) holds. This completes the proof.
\end{proof}
The following Corollary is an analogue of Corollary 1 of \cite{ZLY} which states that if 
$(L,\gamma)$ is a Lie co-algebra; and if $V$ and $W$ are two comodules over $L$, with structure maps
$\Delta_V$ and $\Delta_W$ respectively.  Then $V\otimes W$ is a comodule over $L$
with structure map $\Delta_{V\otimes W}=\Delta_V\otimes I+(\tau\otimes I)(I\otimes \Delta_W)$. A 
direct consequence is that $L\otimes L$ is a module over $L$.

\begin{corollary}
 Let $(L, \gamma, \alpha)$ be a Hom-Lie coalgebra. Then $(L\otimes L, \alpha\otimes\alpha)$ is a 
 comodule over $(L, \gamma, \alpha)$.
\end{corollary}
% {\color{red} Quelle est l'utilit\'e des r\'esultats suivants
% \begin{corollary}[\cite{ZLY}]
%  Let $(L,\gamma)$ be a Lie coalgebra. Let $V$ and $W$ be two comodules over $L$, with structure maps
% $\Delta_V$ and $\Delta_W$ respectively.  Then $V\otimes W$ is a {\color{red}module} over $L$ with structure map 
% $\Delta_{V\otimes W}=\Delta_V\otimes I+(\tau\otimes I)(I\otimes \Delta_W)$.
% \end{corollary}
% \begin{corollary}[\cite{ZLY}]
%  Let $(L, \gamma)$ be a Lie coalgebra. Then $L\otimes L$ is a {\color{red}module} over $L$.
% \end{corollary}
% }
 Let us now give the following definition.
\begin{definition}
A non-empty subset $M'$ of a comodule $(M,\beta_M)$ over the Hom-Lie coalgebra $(L,\gamma,\alpha)$, with structure map $\Gamma$, 
is said to be a subcomodule of $M$, if $\beta_M(M')\subset M'$ and $\Gamma({M'})\subset L\otimes M'$.
\end{definition}
% \begin{definition}
%  Morphism of comodule ...
% \end{definition}
\begin{proposition}
 Let $(L, \gamma, \alpha)$ be a Hom-Lie coalgebra, $V$ be an $L$-comodule. Then, there exists a direct sum decomposition of
 Hom-Lie comodules as 
$V\otimes V=V_+\oplus V_-$, where $V_+=\{s\in V\otimes V : \tau(s)=s\}$ and $V_-=\{s\in V\otimes V :  \tau(s)=-s\}$.
\end{proposition}
\begin{proof}
 It follows from the proof of Proposition 2 in \cite{ZLY} by observing that $V_+$ and $V_-$ are subcomodules of $V\otimes V$.
\end{proof}

\subsection{Hom-Poisson comodules}
% {\color{red} Y a un probl\`eme avec la d\'efinition suivante. C'est quoi {\bf L} ?}
\begin{definition}\label{hpm}
 Let $(A, \delta, \gamma, \alpha)$ be a cocommutative Hom-Poisson coalgebra and  $(M, \beta)$ be a Hom-module. 
 A Hom-Poisson comodule on $M$
 consists of two linear maps $\Delta : M \rightarrow A\otimes M$ and $\Gamma : M \rightarrow A\otimes M$ such that
 $(M, \Delta, \beta)$ is a comodule over the Hom-coassociative coalgebra $(A, \delta, \alpha)$ and
$(M, \Gamma, \beta)$ is a comodule over the Hom-Lie coalgebra $(A, \gamma, \alpha)$ satisfying 
%  $A$-comodule and a {\color{red} $L$}-comodule satisfying the following two relations:
\begin{eqnarray}
(\alpha\otimes \Delta)\circ\Gamma&=&(\gamma\otimes\beta)\circ\Delta+(\tau\otimes Id_M)\circ(\alpha\otimes\Gamma)\circ\Delta,\label{pm21} \\ 
(\delta\otimes\beta)\circ\Gamma&=&(\alpha\otimes\Gamma)\circ\Delta+(\tau\otimes Id_M)\circ(\alpha\otimes\Gamma)\circ\Delta. \label{pm22}
\end{eqnarray}
The maps $\Delta$ and $\Gamma$ will be called the structure maps of the Hom-Poisson comodule and the quadruple $(M,\Delta,\Gamma,\beta)$
will be called a Hom-Poisson comodule over $(A, \delta, \gamma, \alpha)$.
%{\color{red} Ici, $\delta$ n'a jou\'e aucun r\^ole.}
\end{definition}
\begin{remark}
  The conditions (\ref{pm21}) and (\ref{pm22}) can be respectively written as
\begin{eqnarray}
 \alpha(m_{[-1]})\!\otimes\! m_{[0](-1)}\!\otimes\ m_{[0](0)}\!=\!m_{(-1)(1)}\!\otimes\! m_{(-1)(2)}\!\otimes\!\beta(m_{(0)})
 +m_{(0)[-1]}\!\otimes\!\alpha(m_{(-1)})\!\otimes\! m_{(0)[0]},\label{pm21a}\\
m_{[-1]1}\otimes m_{[-1]2}\otimes\beta(m_{[0]})=\alpha(m_{(-1)})\otimes m_{(0)[-1]}\otimes m_{(0)[0]}
+m_{(0)[-1]}\otimes \alpha(m_{(-1)})\otimes m_{(0)[0]},\label{pm21b}
\end{eqnarray}
where we set $\delta(x)=x_1\otimes x_2$, $\gamma(x)=x_{(1)}\otimes x_{(2)}$,
 $\Delta(m)=m_{(-1)}\otimes m_{(0)}$  and  $\Gamma(m)=m_{[-1]}\otimes m_{[0]}$ 
 (without summation symbol, for symplicity).
\end{remark}
\begin{definition}
 %A morphism of Hom-Poisson comodules is a morphism of both underlying associative and Lie comodule structures.
A morphism of Hom-Poisson comodules is a morphism of both underlying comodule structures over 
the Hom-coassociative coalgebras and  the Hom-Lie coalgebras.
 \end{definition}
 \begin{proposition}
 When $(M, \Delta, \Gamma, \beta)$ is a comodule over the cocommutative Hom-Poisson coalgebra 
 $(A, \delta, \gamma, \alpha)$, then $(M, -\Delta, -\Gamma, \beta)$ is a comodule over the 
 cocommutative Hom-Poisson coalgebra $(A, -\delta, -\gamma, \alpha)$.
\end{proposition}
% \begin{proof}
%  ....................................................................
%  
%  \vskip 5cm 
%  
%  
% \end{proof}
\begin{theorem}%\label{}
Let $(A, \delta, \gamma, \alpha)$ be a cocommutative Hom-Poisson coalgebra and $(M, \Delta, \Gamma, \beta)$ 
be a Hom-Poisson comodule. Then 
\begin{equation}
 \tilde\Delta=(\alpha^2\otimes Id_M)\circ\Delta :  M \longrightarrow A\otimes M, \quad 
 \tilde\Gamma=(\alpha^2\otimes Id_M)\circ\Gamma : M \longrightarrow A\otimes M, \label{eq:delta-gamma-tilde}
\end{equation}
% \begin{eqnarray}
%  \tilde\Delta&=&(\alpha^2\otimes Id_M)\circ\Delta :  M \longrightarrow A\otimes M, \label{eq:delta-tilde}\\
% \tilde\Gamma&=&(\alpha^2\otimes Id_M)\circ\Gamma : M \longrightarrow A\otimes M, \label{eq:gamma-tilde}
% \end{eqnarray}
are structure maps of another Hom-Poisson comodule structure on $M$.
\end{theorem}
\begin{proof} 
First of all, we set $Id_M=I$. For any $m\in M$, we have on one hand :
\begin{eqnarray}
(\alpha\otimes\tilde\Delta)\circ\tilde\Gamma(m)\!\!\!\!\!&=&\!\!\!\!\!(\alpha\otimes\tilde\Delta)\circ(\alpha^2\otimes I)\circ\Gamma(m)\nonumber\\
&=&\!\!\!\!\! (\alpha\otimes\tilde\Delta)\circ(\alpha^2\otimes I)(m_{[-1]}\otimes m_{[0]}) \cr 
&=&\!\!\!\!\!\alpha^3(m_{[-1]})\otimes[(\alpha^2\otimes I)\circ\Delta(m_{[0]})]\nonumber\\
&=&\!\!\!\!\!\alpha^3(m_{[-1]})\otimes\alpha^2(m_{[0](-1)})\otimes m_{[0](0)}\nonumber\\
&=&\!\!\!\!\!\Big((\alpha^2)^{\otimes2}\otimes I\Big)\Big(\alpha(m_{[-1]})\otimes m_{[0](-1)}\otimes m_{[0](0)}\Big)\nonumber\\
&=&\!\!\!\!\!\!\Big(\!(\alpha^2)^{\otimes2}\!\otimes\! I\!\Big)\Big(\!m_{(-1)(1)}\!\otimes\! m_{(-1)(2)}\!\otimes\!\beta(m_{(0)})\!
+\!m_{(0)[-1]}\!\otimes\!\alpha(m_{(-1)})\!\otimes\! m_{(0)[0]}\!\Big)
\mbox{ {(by (\ref{pm21a}))}} \nonumber\\
&=&\!\!\!\!\!\alpha^2\big(m_{(-1)(1)}\big)\otimes\alpha^2\big(m_{(-1)(2)}\big)\otimes\beta\big(m_{(0)}\big)
+\alpha^2\big(m_{(0)[-1]}\big)\otimes\alpha^3\big(m_{(-1)}\big)\otimes m_{(0)[0]}\nonumber\\
&=&\!\!\!\!\!\Big[(\alpha^2)^{\otimes2}\!\Big(m_{(-1)(1)}\!\otimes\! m_{(-1)(2)}\!\Big)\!\Big]\!\otimes\!\beta(m_{(0)})\! %\nonumber\\
%& &\!\!\!\!\!
+\!(\tau\!\otimes\! I)\Big(\alpha^3\!(m_{(-1)})\!\otimes\!\alpha^2\!(m_{(0)[-1]})\!\otimes\! m_{(0)[0]}\Big)\nonumber\\
&=&\!\!\!\!\!\Big[\Big((\alpha^2)^{\otimes2}\!\circ\!\gamma\Big)(m_{(-1)})\Big]\!\otimes\!\beta(m_{(0)})\!%\nonumber\\
%& &\!\!\!\!\!
+\!(\tau\!\otimes\! I)\Big[\alpha^3\!(m_{(-1)})\!\otimes\!(\alpha^2\!\otimes\! I)\Big(m_{(0)[-1]}\!\otimes\! m_{(0)[0]}\Big)\Big]\nonumber\\
&=&\!\!\!\!\!\gamma\Big(\alpha^2(m_{(-1)})\Big)\otimes\beta(m_{(0)})
+(\tau\otimes I)\Big[\alpha^3(m_{(-1)})\otimes(\alpha^2\otimes I)\circ\Gamma(m_{(0)})\Big]\nonumber\\
&=&\!\!\!\!\!(\gamma\otimes\beta)\Big(\alpha^2(m_{(-1)})\otimes m_{(0)}\Big)
+(\tau\otimes I)\Big(\alpha^3(m_{(-1)})\otimes\tilde\Gamma(m_{(0)})\Big)\nonumber\\
&=&\!\!\!\!\!(\gamma\otimes\beta)\circ(\alpha^2\otimes I)\circ\Delta(m)
+(\tau\otimes I)(\alpha\otimes\tilde\Gamma)(\alpha^2\otimes I)\Delta(m)\nonumber\\
%&=&\!\!\!\!\!(\gamma\otimes\beta)\tilde\Delta(m)+(\tau\otimes I)(\alpha\otimes\tilde\Gamma)(\alpha^2\otimes I)\Delta(m)\nonumber\\
&=&\!\!\!\!\!(\gamma\otimes\beta)\circ\tilde\Delta(m)+(\tau\otimes I)\circ(\alpha\otimes\tilde\Gamma)\circ\tilde\Delta(m);\nonumber
\end{eqnarray}
and on the other hand :
\begin{eqnarray}
(\delta\otimes\beta)\circ\tilde\Gamma(m) \!\!\!\!\!&=&\!\!\!\!\! (\delta\otimes\beta)\circ(\alpha^2\otimes I)\circ\Gamma(m)\nonumber\\
&=&\!\!\!\!\!(\delta\otimes\beta)\Big(\alpha^2(m_{[-1]})\otimes m_{[0]}\Big) \cr 
&=&\!\!\!\!\!\Big[\delta\Big(\alpha^2(m_{[-1]})\Big)\Big]\otimes\beta(m_{[0]})\nonumber\\
&=&\!\!\!\!\! \Big(\alpha^2(m_{[-1]})\Big)_1\otimes\Big(\alpha^2(m_{[-1]})\Big)_2\otimes\beta(m_{[0]})\cr 
&=&\!\!\!\!\!\alpha^2(m_{[-1]1})\otimes\alpha^2(m_{[-1]2})\otimes\beta(m_{[0]}) \cr 
&=&\!\!\!\!\!\Big((\alpha^2)^{\otimes2}\otimes I\Big)\Big(m_{[-1]1}\otimes m_{[-1]2}\otimes\beta(m_{[0]})\Big)\nonumber\\
&=&\!\!\!\!\!\Big(\!(\alpha^2)^{\otimes2}\!\otimes\! I\!\Big)\Big(\!\alpha(m_{(-1)})\!\otimes\! m_{(0)[-1]}\!\otimes\! m_{(0)[0]} %\nonumber\\
%&&\!\!\!\!\!
\!+\! m_{(0)[-1]}\!\otimes\!\alpha(m_{(-1)})\!\otimes\! m_{(0)[0]}\!\Big) \text{(by (\ref{pm21b})) }\nonumber\\
&=&\!\!\!\!\!\alpha^3(m_{(-1)})\otimes\alpha^2\Big(m_{(0)[-1]}\Big)\otimes m_{(0)[0]}+\alpha^2\Big(m_{(0)[-1]}\Big)\otimes\alpha^3(m_{(-1)})\otimes m_{(0)[0]}\nonumber\\
&=&\!\!\!\!\!\alpha^3(m_{(-1)})\otimes(\alpha^2\otimes I)\circ\Gamma(m_{(0)})%\nonumber\\
%& &\!\!\!\!\!
+(\tau\otimes I)\Big(\alpha^3(m_{(-1)})\otimes\alpha^2(m_{(0)[-1]})\otimes m_{(0)[0]}\Big)\nonumber\\
&=&\!\!\!\!\!\alpha^3(m_{(-1)})\otimes\tilde\Gamma(m_{(0)})+(\tau\otimes I)\Big[\alpha^3(m_{(-1)})\otimes\Big((\alpha^2\otimes I)\circ\Gamma(m_{(0)})\Big)\Big]\nonumber\\
&=&\!\!\!\!\!(\alpha\otimes\tilde\Gamma)\Big(\alpha^2(m_{(-1)})\otimes m_{(0)}\Big)+(\tau\otimes I)\Big(\alpha^3(m_{(-1)})\otimes\tilde\Gamma(m_{(0)})\Big)\nonumber\\
&=&\!\!\!\!\!(\alpha\otimes\tilde\Gamma)\circ(\alpha^2\otimes I)\circ\Delta(m)
+(\tau\otimes I)\circ(\alpha\otimes\tilde\Gamma)\Big(\alpha^2(m_{(-1)})\otimes m_{(0)}\Big)\nonumber\\
&=&\!\!\!\!\!(\alpha\otimes\tilde\Gamma)\circ\tilde\Delta(m)+(\tau\otimes I)\circ(\alpha\otimes\tilde\Gamma)\circ(\alpha^2\otimes I)\circ\Delta(m)\nonumber\\
&=&\!\!\!\!\!(\alpha\otimes\tilde\Gamma)\circ\tilde\Delta(m)+(\tau\otimes I)\circ(\alpha\otimes\tilde\Gamma)\circ\tilde\Delta(m)\nonumber.
\end{eqnarray}
Therefore, we have another Hom-Poisson comodule structure on $M$.
\end{proof}
In the following two results, $\tilde\Delta$ and $\tilde\Gamma$ are given by  (\ref{eq:delta-gamma-tilde}).
\begin{corollary}
Let $(A, \delta, \gamma, \alpha)$ be a cocommutative Hom-Poisson coalgebra. 
If $(M, \Delta, \Gamma, \beta)$ is a Hom-Poisson comodule over the  cocommutative Hom-Poisson coalgebra 
$(A, \delta\circ\alpha, \gamma\circ\alpha, \alpha)$, 
then $(M, \tilde\Delta, \tilde\Gamma, \beta)$ is a comodule over $(A, \delta, \gamma, \alpha)$.
\end{corollary}
% \begin{proof}
%   ............................................................
%   
%   \vskip 3cm
% \end{proof}
%{\color{red} C'est quoi $\tilde f$ ? Il me semble que c'est simplement $f$.}
\begin{proposition}
Let $f : (M, \Delta, \Gamma, \beta)\rightarrow (M', \Delta', \Gamma', \beta')$ be a morphism of modules over the Hom-Poisson coalgebra
 $(A, \delta, \gamma, \alpha)$.
Then, $f: (M, \tilde\Delta, \tilde\Gamma, \beta)\rightarrow (M', \tilde\Delta', \tilde\Gamma', \beta')$
is a morphism of comodules over the Hom-Poisson coalgebra $(A, \delta, \gamma, \alpha)$.
\end{proposition}
\begin{proof}
  The proof is similar to that of Theorem \ref{morfi}.
%  {\color{red} Ecrire quand m\^eme la preuve}
\end{proof}
\noindent
 {\bf Acknowledgements.} {
The second author is supported by the {\bf nlaga project} and this work was done during 
his visit at IMSP (Port-Novo, Benin) funded by the {\bf Deutscher Akademischer Austausch Dienst (DAAD)}. To both institutions he   
expresses his gratitude and thanks.}

 %{\bf Received: November, 2014}


\begin{thebibliography}{99} 

\bibitem{arnold-math-method-classical-mech}  
Arnold V. I.,  
{\it Mathematical methods of classical mechanics}, 
Grad. Texts in Math. 60, Springer, Berlin, 1978.

\bibitem{bakayoko:module-color-hom-poisson} 
Bakayoko I., 
{\it Modules over color Hom-Poisson algebras}, 
J. Gen. Lie Theory Appl. Vol. 8, Num. 1 (2014), 1-6.

\bibitem{BI} Bakayoko I., 
{\it $L$-modules, $L$-comodules and Hom-Lie quasi-bialgebras}, 
Afr.  Diaspora J. Math., Vol. 17  (2014), 49-64.

\bibitem{bordeman-elchinger-makhlouf:twisting-poisson-copoisson-quantization} 
Bordemann M.; Elchinger O.; Makhlouf A., 
{\it Twisting Poisson algebras, coPoisson algebras and Quantization}, 
arXiv:13072612v1, May 3, 2012.

\bibitem{drinfeld:quantum-group}  
Drinfel'd V. G., 
{\it Quantum groups}, 
in Proc. ICM (Berkeley, 1986), p.798-820, AMS Providence, RI, 1987.

\bibitem{elhamdadi-makhlouf:deformations-hom-alternative-hom-malcev}
Elhamdadi M. and Makhlouf A., 
{\it Deformations of Hom-alternative and Hom-Malcev algebras}, 
arXiv :1006.2499, June 12, 2010. 

\bibitem{frenkel-ben_zvi:vertex} 
Frenkel E. and Ben-Zvi D., 
{\it Vertex algebras and algebraic curves}, 
Math. Surveys and Monographs 88, 2nd ed., AMS, Providence,  RI, 2004.

\bibitem{gerstenhaber-deformation-ring-algebra} 
Gerstenhaber M., 
{\it On the deformation of rings and algebras}, 
Ann. Math. 79 (1964), 59-103.

\bibitem{hartwig-larson-silvestrov:deformations} 
Hartwig J. T.;  Larsson D.  and  Silvestrov S. D., 
{\it  Deformations of Lie algebras using $\sigma$-derivations}, 
J. Algebra 295 (2006), no. 2, 314-361.

\bibitem{kontsevich:deformation-quatization-poisson} 
Kontsevich M., 
{\it Deformations quantization of Poisson manifolds},
Lett. Math. Phys. 66 (2003), 157-216.

\bibitem{larson:global-arithmetic-hom-lie-algebras}
Larsson D., 
{\it Global and arithmetic Hom-Lie algebras}, 
Uppsala Univ. UUDM. Report 2008:44, 
\url{http://www.math.uu.se/research/pub/preprints.php}.


\bibitem{larson-silvestrov:quasi-hom-lie-central-extensions} 
Larsson D.  and  Silvestrov S. D., 
{\it  Quasi-hom-Lie algebras, central extensions and 2-cocycle-like identities}, 
J. Algebra 288 (2005), no. 2, 321–344.

\bibitem{makhlouf:hom-alternative-hom-jordan} 
Makhlouf A., 
{\it Hom-alternative algebras and Hom-Jordan algebras}, 
Int. Electron. J. Algebra, Vol. 8 (2010), 177-190.

\bibitem{AS2} Makhlouf A.; Silvestrov S., 
{\it Notes on Formal Deformations of  Hom-associative algebras and Hom-Lie algebras},
Forum Math., 22 (2010), no. 4, 715-739.

\bibitem{shaller-strobl:poisson-structure} 
Shaller P.;  Strobl T., 
{\it Poisson structure induced (topological) field theories}, 
Mod. Phys. Lett. A 9 (1994), 3129-3136.

\bibitem{vaisman-lecture-geometry-poisson} 
Vaisman I.,  
{\it Lecture on the geometry of Poisson manifolds}, 
Birkausser, Bassel, 1994.

\bibitem{yau:noncom-hom-poisson-algebras} 
Yau D., 
{\it Noncommutative Hom-Poisson algebras}, 
ArXiv:1010.3408v1,  October 17, 2010.

\bibitem{yau-hom-bialgebras-comodule-hom-bialgebras}
Yau D., 
{\it Hom-bialgebras and comodule Hom-algebras}
Int. Electron. J. Algebra,
Vol. 8 (2010), 45-64.

\bibitem{ZLY} 
Yun Z. L., 
{\it Lie comodules and the construction of Lie bialgebras}, 
Science in China Press, Jun 2008, Vol. 51, No 6, 1017-1026.

\bibitem{zhang:comodule-hom-coalgebras} 
Zhang T.,   
{\it comodule Hom-coalgebras}, 
ArXiv:1301.4152,  January 17, 2013.

\end{thebibliography}
\end{document}